\documentclass[reqno]{amsart}
\usepackage{amscd}

\def\beq{\begin{equation}}
\def\eeq{\end{equation}}
\def\ba{\begin{array}}
\def\ea{\end{array}}

\def\R{\mathbb R}

\newcommand{\I}{{\mathbf I}}

\newcommand{\K}{{\mathcal K}}

\newtheorem{thm}{Theorem}[section]
\newtheorem{lm}[thm]{Lemma}

\newtheorem{crl}[thm]{Corollary}

\theoremstyle{definition}
\newtheorem{rem}[thm]{Remark}

\newtheorem{df}[thm]{Definition}

\theoremstyle{remark}

\begin{document}
\pagestyle{plain}
\title{Solutions to discrete fractional Sch\"{o}dinger equations}


\author{Lidan Wang}
\email{lidanwang2023@126.com}
\address{Lidan Wang: School of Mathematical Sciences, Fudan University, Shanghai 200433, China}


\begin{abstract}
In this paper, we study the discrete fractional Schr\"{o}dinger equation $$ (-\Delta)^\alpha u+h(x) u=f(x,u),\quad x\in \mathbb{Z}^d,$$
 where $d\in\mathbb{N}^*,\,\alpha \in(0, 1)$ and the nonlocal operator $(-\Delta)^\alpha $ is defined by discrete Fourier transform, which differs from the continuous case. Under suitable assumptions on $h$ and $f$, we prove the existence and multiplicity of solutions to this equation by variational method.

\end{abstract}

 \maketitle



\section{Introduction}
 Given $\alpha\in(0,1),$ consider the fractional Sch\"{o}dinger equation
 \begin{eqnarray} \label{00}
 	(-\Delta)^\alpha u+h(x)u=f(x,u),\quad x\in\mathbb{R}^d,
  \end{eqnarray}
  where $(-\Delta)^\alpha$
is a nonlocal operator that we may define in several ways. For example, it can be defined as 
\begin{equation}\label{vj}
(-\Delta)^\alpha u(x)=C(\alpha,d) P.V.\int_{\mathbb{R}^d}\frac{u(x)-u(y)}{|x-y|^{d+2\alpha}}\,dy,\quad u\in\mathcal{S},	
\end{equation}
where $C(\alpha,d)$ is a normalized constant, $P.V.$ stands for the Cauchy principle value and $\mathcal{S}$ is the Schwartz space of rapidly decaying functions. Moreover, it can also be described by means of Fourier transform, that is
\begin{equation}\label{vh}
\mathcal{F}((-\Delta)^\alpha u)(\xi)=|\xi|^{2\alpha}\mathcal{F}((-\Delta)^\alpha u)(\xi),\quad \xi\in \mathbb{R}^d,	
\end{equation}
where $\mathcal{F}$ denotes the Fourier transform. We would like to mention that different definitions  of the fractional Laplace operator $(-\Delta)^\alpha$ are all equivalent. For an elementary introduction to the fractional Laplace operator and fractional
Sobolev space, we refer the interested readers to \cite{CLM,DP}.

The equation (\ref{00}) was first introduced by Laskin \cite{L,L0} as a result of expanding the Feynman
path integral from the Brownian-like to the L\'{e}vy-like quantum mechanical paths.  After that, many contributions
have appeared on the study of solutions to this equation. For example, if the operator $(-\Delta)^{\alpha}$ is given by the Fourier transform (\ref{vh}), 
under the  coercivity assumption on $h$, i.e. $h(x)\rightarrow+\infty$ as $|x|\rightarrow+\infty$, Cheng \cite{C} proved the  existence of ground state solutions to the equation (\ref{00}) with $f(x,u)=|u|^{p-2}u$ by Lagrange multiplier method. Later, Secchi \cite{S0} extended the results of Cheng \cite{C} to general nonlinearity  $f(x,u)$ which is differential
and $f$ satisfies Ambrosetti-Rabinowitz condition. It is worthwhile to
remark that in \cite{C} and \cite{S0} the coercivity hypothesis  is assumed on $h$ in order to overcome the problem
of lack of compactness, typical of elliptic problems defined in unbounded domains. For more works about the existence and multiplicity results to the equation (\ref{00}) obtained by variational and topological methods, we refer to \cite{CW,CW1,DA,DP1,FQT,ML,SZ,S1,WZ,ZX} and the references therein. 
  
 Nowadays, a great attention has been focused on the discrete fractional Laplace operators and related fractional problems.  In \cite{LR}, the authors defined the fractional Laplace operator $(-\Delta)^{\alpha}(0<\alpha<1)$ on $\mathbb{Z}^d$  with the semigroup method (see \cite{S2,ST}) as
\begin{equation}\label{vq}
(-\Delta)^{\alpha}u(x)=\frac{1}{\Gamma(-\alpha)}\int^{\infty}_{0}(e^{t\Delta}u(x)-u(x))\frac{dt}{t^{1+\alpha}},\quad u\in\ell^\infty(\mathbb{Z}^d).	
\end{equation}
Here $\Gamma$ represents the Gamma function, $v(x,t)=e^{t\Delta}u(x)$ is the solution of the semidiscrete heat equation
\begin{eqnarray*}
    \left\{
    \begin{array}{rcll}
       \partial_t v(x,t)&=&\Delta v(x,t),  & \text{in~}\mathbb{Z}^d\times(0,\infty), \\
        v(x,0)&=&u(x), & \text{on~}\mathbb{Z}^d,
    \end{array}
    \right.
\end{eqnarray*}
and 
$$\Delta u(x)=\sum\limits_{i=1}^d\left(u(x+e_i)-2u(x)+u(x-e_i)\right),$$
where $e_i$ denotes the unit vector in the positive direction of the $i$-th coordinate.
 Furthermore, if $d=1$, 
by the definition (\ref{vq}), Ciaurri et al. \cite{CR} established  the pointwise nonlocal formula for $(-\Delta)^{\alpha}$, i.e.
  \begin{equation}\label{vw}
  	(-\Delta)^{\alpha}u(x)=2\sum\limits_{x, y\in\mathbb{Z}, y\neq x}\left(u(x)-u(y)\right)K_\alpha(x-y),\quad u\in\ell_\alpha,
  \end{equation}
 where the discrete kernel $K_\alpha$ satisfies
 $$K_\alpha(0)=0,\quad \text{and}\quad \frac{c_\alpha}{|x|^{1+2\alpha}}\leq K_\alpha(x)\leq \frac{C_\alpha}{|x|^{1+2\alpha}},\quad x\in\mathbb{Z}\backslash\{0\},$$ 
  and $$\ell_\alpha=\left\{u:\mathbb{Z}\rightarrow\mathbb{R}|\sum\limits_{y\in\mathbb{Z}}\frac{|u(y)|}{(1+|y|^{1+2\alpha})}<+\infty\right\}.$$
  Actually, the formula (\ref{vw}) can be considered as a discrete analog of (\ref{vj}) with $d=1$.
  
 Based on the definition (\ref{vw}), Xiang and Zhang \cite{XZ} considered the nonlinear discrete fractional Schr\"{o}dinger equation 
 $$(-\Delta)^\alpha u+h(x)u=\lambda f(x,u),\quad x\in \mathbb{Z}.
$$
Under the coercivity assumption on $h$ ($h(x)\rightarrow+\infty$ as $|x|\rightarrow+\infty$) and suitable assumptions on the nonlinearity $f$,  the authors \cite{XZ} obtained two distinct nontrivial and nonnegative homoclinic solutions for large $\lambda$ by the mountain pass theorem and Ekeland's variational principle. Later, under weaker assumption on $h$, Ju and Zhang \cite{JZ} proved the multiplicity of homoclinic solutions to the above equation by Clark's theorem and its variants. Moreover, the same result in  \cite{XZ} was extended to the setting of the discrete fractional $p$-Laplacian by the same approach as in \cite{XZ}, see \cite{WT} for details. For more related works on integers, we refer readers to \cite{AB,CG,CL,HO,JD}.
  
  Recently, the authors in \cite{LM,LR} introduced the Fourier transform on lattice graph $\mathbb{Z}^d$, namely, for a given function $u\in\ell^1(\mathbb{Z}^d)$, the discrete Fourier transform
is defined by
 $$\mathcal{F}(u)(\theta)=\sum\limits_{x\in\mathbb{Z}^d} u(x)e^{ix\cdot\theta},\quad \theta\in[-\pi,\pi]^d,$$
and the inverse discrete Fourier transform is given by 
\begin{eqnarray*}
\mathcal{F}^{-1}(v)(x)=\frac{1}{(2\pi)^d}\int_{[-\pi,\pi]^d}v(\theta) e^{-ix\cdot\theta}\,d\theta,\quad x\in\mathbb{Z}^d,
\end{eqnarray*}
where $x\cdot\theta=\sum\limits_{i=1}^dx_i\theta_i.$  From the definitions above, one sees that the discrete Fourier transform $\mathcal{F}(u)$ is a function defined on $[-\pi,\pi]^d\subset\mathbb{R}^d$ whose Fourier coefficients are given by the sequence $\{u(x)\}_{x\in\mathbb{Z}^d}$, and the operator $u\mapsto\mathcal{F}(u)$  can be extended as an isometry from $\ell^2(\mathbb{Z}^d)$ into $L^2([-\pi,\pi]^d)$. 
Moreover, by \cite{LR}, the discrete fractional Laplacian $(-\Delta)^{\alpha}u$ can be characterized as
\begin{equation}\label{b1}
\mathcal{F}\left((-\Delta)^{\alpha}u\right)(\theta )=\left(\sum\limits_{i=1}^d 4\sin^2 (\theta_i/2)\right)^{\alpha}\mathcal{F}(u)(\theta),\quad \theta\in[-\pi,\pi]^d.
\end{equation}
As a consequence, by the convolution and uniqueness properties of the Fourier transform, we have
\begin{equation}\label{ve}
(-\Delta)^{\alpha}u(x)=\sum\limits_{y\in \mathbb{Z}^d} K^\alpha(x-y)u(y),\quad x\in\mathbb{Z}^d,
 \end{equation}
where $$K^\alpha(x):=\frac{1}{(2\pi)^d}\int_{[-\pi,\pi]^d}\left(\sum\limits_{i=1}^d 4\sin^2 (\theta_i/2)\right)^{\alpha}e^{-ix\cdot\theta}\,d\theta.$$
  
If the operator $(-\Delta)^{\alpha}$ is defined by the Fourier transform (\ref{b1}), to the best of our knowledge, there are no papers considering the discrete  nonlinear fractional Sch\"{o}dinger equations.
Hence motivated by the works above, in this paper, we study the discrete fractional Schr\"{o}dinger equation 
    \begin{equation}\label{01}
    (-\Delta)^\alpha u+h(x)u=f(x,u),\quad x\in \mathbb{Z}^d,    \end{equation}
where $(-\Delta)^{\alpha}u$ is defined by the Fourier transform (\ref{b1}). We always assume that 
\begin{enumerate}
\item[($h_1$)] $c_1\leq h(x)\leq c_2$ for $x\in\mathbb{Z}^d,$ where $c_1$ and $c_2$ are positive constants;
\item[($f_1$)] $f:\mathbb{Z}^d\times \R\rightarrow \R$ is continuous with respect to $u\in\R$;

\item[($f_2$)] $
|f(x,u)|\leq a(1+|u|^{p-1})$ for some $a>0$ and $p>2$;

\item[($f_3$)]  $f(x,u)=o(u)$ uniformly in $x$ as $|u|\rightarrow 0$;
\item[($f_4$)] $u\mapsto \frac{f(x,u)}{|u|}$ is increasing  on $(-\infty, 0)$ and $(0, +\infty)$;
\item[($f_5$)] $\frac{F(x,u)}{u^2}\rightarrow+\infty$ uniformly in $x$ as $|u|\rightarrow +\infty$ with $F(x,u)=\int_{0}^{u}f(x,t)\,dt$;
\item[($f_6$)]  $h$ and $f$ are $1$-periodic in $x_i,\,1\leq i\leq d.$ 
\end{enumerate}
Clearly, by ($f_1$), ($f_2$) and ($f_3$), for any $\varepsilon>0$, there exists $c_\varepsilon>0$ such that
\begin{equation}\label{nn}
|f(x,u)|\leq\varepsilon|u|+c_\varepsilon|u|^{p-1},\quad u\in\R.
\end{equation}
Moreover, by ($f_3$) and ($f_4$), we have 
\begin{equation}\label{30}
f(x,u)u>2 F(x, u)> 0, \quad u\neq0.
\end{equation}
A simple case of $f$ satisfying $(f_1)$-$(f_5)$ is the function $f(x,u)=|u|^{p-2}u.$ Now we sate  our main results as follows.

\begin{thm}\label{t-0} 

Let $(h_1)$ and $(f_1)$-$(f_6)$ be satisfied. Then there exists a ground state solution to the equation (\ref{01}).
\end{thm}

By $(f_6)$, if $u$ is a solution of the equation (\ref{01}), then so is $u(\cdot-y)$ for all $y\in\mathbb{Z}^d$.
Let $$\mathcal{O}(u)=\{u(\cdot-y):y\in\mathbb{Z}^d\}.$$
$\mathcal{O}(u)$ is called the orbit of $u$ under the action of $\mathbb{Z}^d.$ Let $u_1$ and $u_2$ be two solutions of the equation (\ref{01}), we shall call $u_1$ and $u_2$ are geometrically
distinct if $\mathcal{O}(u_1)$ and $\mathcal{O}(u_2)$ are disjoint.

\begin{thm}\label{t-1}
Let $(h_1)$ and $(f_1)$-$(f_6)$ be satisfied. If $f(x,-u)=-f(x,u)$ for $(x,u)\in\mathbb{Z}^d\times\mathbb{R}$, then the equation (\ref{01})	admits infinitely many pairs of geometrically distinct solutions in $H^\alpha$.
\end{thm}

This paper is organized as follows. In Section 2, we recall some basic facts and give some useful lemmas. In Section 3, we prove the existence of ground state solutions to the equation (\ref{01}) by the method of  Nehari manifold developed by Szulkin and Weth \cite{SW}.  In Section 4,  we follow the ideas of Szulkin and Weth \cite{SW1} to prove the multiplicity of solutions to the equation (\ref{01}).

\section{Preliminaries}

In this section, we introduce some settings on graphs and give some useful lemmas.

Let $G=(V,E)$ be a connected, locally finite graph, where $V$ denotes the vertex set and $E$ denotes the edge set. We call vertices $x$ and $y$ neighbors, denoted by $x\sim y$, if there is an edge connecting them, i.e. $(x,y)\in E$.
For any $x,y\in V$, the distance $|x-y|$ is defined as the minimum number of edges connecting $x$ and $y$, i.e.
$$|x-y|=\inf\{k:x=x_0\sim\cdots\sim x_k=y\}.$$

In this paper, we consider the natural discrete model of the Euclidean space, the integer lattice graph. The $d$-dimensional integer lattice graph, denoted by $\mathbb{Z}^d$, consists of the set of vertices $V=\mathbb{Z}^d$ and the set of edges $E=\{(x,y): x,\,y\in V,\,\underset {{i=1}}{\overset{d}{\sum}}|x_{i}-y_{i}|=1\}.$

 We denote the space of real functions on $\mathbb{Z}^d$ by $C(\mathbb{Z}^d)$. 
For any $u\in C(\mathbb{Z}^d)$, the $\ell^p(\mathbb{Z}^d)$ space is given by
$$\ell^s(\mathbb{Z}^d)=\{u\in C(\mathbb{Z}^d):\|u\|_{s}<+\infty\},\qquad s\in[1,+\infty],$$
 where 
 $$\|u\|_{s}=(\sum\limits_{x\in\mathbb{Z}^d}|u(x)|^s)^{\frac{1}{s}},\quad s\in[1,+\infty),\quad \text{and}\quad \|u\|_{\infty}=\underset {x\in\mathbb{Z}^d}{\sup}|u(x)|,\quad s=+\infty.$$
 
As stated in Introduction, the discrete fractional Laplacian can be characterized as
\begin{eqnarray*}
\mathcal{F}\left((-\Delta)^{\alpha}u\right)(\theta )=\left(\sum\limits_{i=1}^d 4\sin^2 (\theta_i/2)\right)^{\alpha}\mathcal{F}(u)(\theta),\quad \theta\in[-\pi,\pi]^d.
\end{eqnarray*}
For convenience, in the following, we denote
$$\Phi(\theta):=\left(\sum\limits_{i=1}^d 4\sin^2 (\theta_i/2)\right)\quad \text{and}\quad\hat{u}:=\mathcal{F}(u).$$

\begin{lm}\label{ok}
If $u\in\ell^2(\mathbb{Z}^d)$, then $[u]_\alpha<+\infty.$ Moreover, we have
$$[u]_\alpha\leq C(\alpha,d)\|u\|_2,$$
where $$[u]^2_\alpha=:\frac{1}{(2\pi)^d}\int_{[-\pi,\pi]^d}(\Phi(\theta))^{\alpha}|\hat{u}(\theta)|^2\,d\theta.$$
	
\end{lm}
\begin{proof}
For $u\in\ell^2(\mathbb{Z}^d)$, by  Plancherel's identity, we have
\begin{eqnarray*}
[u]^2_\alpha\leq C(\alpha,d)\|\hat{u}\|^2_{L^2([-\pi,\pi]^d)}=C(\alpha,d)\|u\|^2_2.
\end{eqnarray*}
\end{proof}
\begin{crl}\label{oo}
If $u\in\ell^2(\mathbb{Z}^d)$, then we have $(-\Delta)^\alpha u\in\ell^2(\mathbb{Z}^d)$ with
$$\|(-\Delta)^\alpha u\|_2\leq C(\alpha, d) \|u\|_2.$$	
\end{crl}




Let 
$$H^\alpha=\left\{u\in\ell^2(\mathbb{Z}^d):\frac{1}{(2\pi)^d}\int_{[-\pi,\pi]^d}(\Phi(\theta))^{\alpha}|\hat{u}(\theta)|^2\,d\theta+\sum\limits_{x\in\mathbb{Z}^d}h(x)|u(x)|^2<+\infty\right\}$$ 
endowed with the norm
$$\|u\|_\alpha=\left([u]^2_\alpha+\sum\limits_{x\in\mathbb{Z}^d}h(x)|u(x)|^2\right)^{\frac{1}{2}}.$$
Moreover, if $H^\alpha$ is equipped with the following inner product
$$(u,v)_\alpha=\frac{1}{(2\pi)^d}\int_{[-\pi,\pi]^d}(\Phi(\theta))^{\alpha}\hat{u}(\theta)\overline{\hat{v}(\theta)}\,d\theta+\sum\limits_{x\in\mathbb{Z}^d}h(x)u(x)v(x),\quad u,v\in H^\alpha,$$
then $H^\alpha$ is a Hilbert space. Now we give a simple proof.

\begin{lm}\label{ou}
Let $(h_1)$ hold. Then $H^\alpha$ is a Hilbert space.	
\end{lm}

\begin{proof}
Clearly, $H^\alpha$ is a linear space. Note that, for $u,v\in H^\alpha,$ $$\left(\widehat{(-\Delta)^{\alpha}u},\hat{v}\right)_{L^2([-\pi,\pi]^d)}=\overline{\left(\hat{v},\widehat{(-\Delta)^{\alpha}u}\right)}_{L^2([-\pi,\pi]^d)}.$$ 
 Then the bilinear form $(u,v)_\alpha$ is symmetric and $(u,u)_\alpha\geq 0.$ By $(h_1)$, we have
\begin{eqnarray}\label{om}
	c_1\|u\|^2_2\leq \sum\limits_{x\in\mathbb{Z}^d} h(x)|u(x)|^2\leq\|u\|^2_\alpha.
\end{eqnarray}
Hence if $(u,u)_\alpha=0$, then $u=0$.

To show completeness, let $\{u_n\}$ be a Cauchy sequence in $H^\alpha$. Then by (\ref{om}), one gets that $\{u_n\}$ is a Cauchy sequence in $\ell^2(\mathbb{Z}^d)$. Hence, there exists $u\in\ell^2(\mathbb{Z}^d)$ such that
\begin{equation}\label{ob}
	u_n\rightarrow u,\quad \text{in~}\ell^2(\mathbb{Z}^d).
\end{equation}
Moreover, by $(h_1)$ and Lemma \ref{ok}, we obtain that
\begin{eqnarray}\label{on}
	\|u\|^2_\alpha\nonumber&=&[u]^2_\alpha+\sum\limits_{x\in\mathbb{Z}^d} h(x)|u(x )|^2\nonumber\\&\leq& C(\alpha,d)\|u\|^2_2+c_2\|u\|^2_2\nonumber\\&=&C_1\|u\|^2_2.
\end{eqnarray}
Hence it follows from (\ref{ob}) and (\ref{on}) that
$u_n\rightarrow u$ strongly in $H^\alpha.$
	
\end{proof}

\begin{rem}\label{os}
	By (\ref{om}) and (\ref{on}), one sees that the norms $\|\cdot\|_\alpha$ and $\|\cdot\|_2$ are equivalent.
	
\end{rem}

\begin{lm}\label{l01}
Let $(h_1)$ hold. Then for any $2\leq q\leq+\infty$, $H^\alpha\subset \ell^q(\mathbb{Z}^d).$
\end{lm}

\begin{proof}
For $2\leq q\leq+\infty$, the result follows from $H^\alpha\subset \ell^2(\mathbb{Z}^d)\subset\ell^q(\mathbb{Z}^d)$.
\end{proof}

The energy functional related to the equation (\ref{01}) is given by
\begin{eqnarray*}
I(u)&=&\frac{1}{2(2\pi)^d}\int_{[-\pi,\pi]^d}(\Phi(\theta))^{\alpha}|\hat{u}(\theta)|^2\,d\theta+\frac{1}{2}\sum\limits_{x\in\mathbb{Z}^d}h(x)|u(x)|^2-\sum\limits_{x\in\mathbb{Z}^d}F(x,u)\\&=&\frac{1}{2}\|u\|^2_\alpha-\sum\limits_{x\in\mathbb{Z}^d}F(x,u).	
\end{eqnarray*}

\begin{lm}\label{vn}
Let $(h_1)$ hold. Then
\begin{itemize}
	\item [(i)] $\phi\in C^1(H^\alpha,\mathbb{R})$ and,
	for any $u,v\in H^\alpha,$ $$\langle \phi'(u),v\rangle=\frac{1}{(2\pi)^d}\int_{[-\pi,\pi]^d}(\Phi(\theta))^{\alpha}\hat{u}(\theta)\overline{\hat{v}(\theta)}\,d\theta+\sum\limits_{x\in\mathbb{Z}^d}h(x)u(x)v(x),$$
	where $\phi(u):=\frac{1}{2}\|u\|^2_\alpha;$
	
	\item[(ii)] $\phi'$ is bounded on bounded sets and $\langle\phi'(w),w\rangle=1$ for $w\in S_1,$ where $S_1$ is the unit sphere in $H^\alpha.$
\end{itemize}

\end{lm}
\begin{proof}
(i) By Lemma \ref{ok}, one sees that $\phi$   is well defined on $H^\alpha.$	Let $t>0$ and fix $u,v\in H^\alpha$, 
\begin{eqnarray*}
	\phi(u+tv)-\phi(u)&=&\frac{1}{2}\|u+tv\|^2_\alpha-\frac{1}{2}\|u\|^2_\alpha\\&=&\frac{1}{2(2\pi)^d}\int_{[-\pi,\pi]^d}(\Phi(\theta))^{\alpha}\left (|\hat{u}(\theta)+t\hat{v}(\theta)|^2-|\hat{u}(\theta)|^2\right)\,d\theta\\&&+\frac{1}{2}\sum\limits_{x\in\mathbb{Z}^d}h(x)\left(|u(x)+tv(x)|^2-|u(x)|^2\right)\\&=&\frac{1}{2(2\pi)^d}\int_{[-\pi,\pi]^d}(\Phi(\theta))^{\alpha}\left (2t\hat{u}(\theta)\overline{\hat{v}(\theta)}+t^2|\hat{v}(\theta)|^2\right)\,d\theta\\&&+\frac{1}{2}\sum\limits_{x\in\mathbb{Z}^d}h(x)\left(2tu(x)v(x)+t^2|v(x)|^2\right).
\end{eqnarray*}
Hence 
\begin{eqnarray*}
	\langle \phi'(u),v\rangle&=&\lim\limits_{t\rightarrow0^+}\frac{\phi(u+tv)-\phi(u)}{t}\\&=&\frac{1}{(2\pi)^d}\int_{[-\pi,\pi]^d}(\Phi(\theta))^{\alpha}\hat{u}(\theta)\overline{\hat{v}(\theta)}\,d\theta+\sum\limits_{x\in\mathbb{Z}^d}h(x)u(x)v(x).
\end{eqnarray*}
Thus $\phi$ is Gateaux differential on $H^\alpha.$ Next we claim that $\phi': H^\alpha\rightarrow(H^\alpha)^*$ is continuous.

In fact, we assume that $\{u_n\}\subset H^\alpha$ is a sequence satisfying $u_n\rightarrow u$ in $H^\alpha$ as $n\rightarrow+\infty.$ By the H\"{o}der inequality, Lemma \ref{ok} and Remark \ref{os}, we have
\begin{eqnarray*}
\|\phi'(u_n)-\phi'(u)\|_{(H^\alpha)^*} &=&\sup\limits_{\|v\|_\alpha\leq 1}|\langle\phi'(u_n)-\phi'(u),v\rangle|\\&\leq&	\sup\limits_{\|v\|_\alpha\leq 1}\left([u_n-u]_\alpha[v]_\alpha+\|u_n-u\|_\alpha\|v\|_\alpha\right)\\&\leq& C(\alpha,d)\|u_n-u\|_\alpha\\&\rightarrow& 0.
\end{eqnarray*}
This implies that $\phi'$ is continuous. Hence $\phi\in C^1(H^\alpha,\mathbb{R})$ .

(ii) Let $w\in S_1$, we have
\begin{eqnarray*}
	\|\phi'(w)\|_{(H^\alpha)^*} &=&\sup\limits_{\|v\|_\alpha\leq 1}|\langle\phi'(w),v\rangle|\\&\leq&	\sup\limits_{\|v\|_\alpha\leq 1}\left([w]_\alpha[v]_\alpha+\|w\|_\alpha\|v\|_\alpha\right)\\&\leq& C(\alpha,d).\end{eqnarray*}
By (i), one gets that
$$\langle \phi'(w),w\rangle=1.$$

\end{proof}
\begin{rem}\label{vb}
It follows from Lemma \ref{vn} that $S_1$ is a $C^1$-submanifold of $H^\alpha$ and the tangent space of $S_1$ at $w$ is
$$T_w(S_1)=\{z\in H^\alpha: \langle\phi'(w),z\rangle=0\}.$$
\end{rem}

By Lemma \ref{vn} and assumptions on $f$, one gets easily that $I\in C^1(H^\alpha,\R)$ and, for $u, v\in H^\alpha$, 
\begin{eqnarray*}
	\langle I'(u),v\rangle&=&\langle \phi'(u),v\rangle-\sum\limits_{x\in\mathbb{Z}^d}f(x,u)v(x)\\&=&(u,v)_\alpha-\sum\limits_{x\in\mathbb{Z}^d}f(x,u)v(x).
\end{eqnarray*}

\begin{lm}\label{oa}
Let $(h_1)$ and $(f_1)$-$(f_5)$ hold. Then every critical point $u\in H^\alpha$	 of the energy functional $I$ is a solution of the equation (\ref{01}).
\end{lm}

\begin{proof}
Let $u\in H^\alpha$ be a critical point of $I$, i.e. $I'(u)=0.$	 Then for any $v\in H^\alpha$,
\begin{equation}\label{oe}
\frac{1}{(2\pi)^d}\int_{[-\pi,\pi]^d}(\Phi(\theta))^{\alpha}\hat{u}(\theta)\overline{\hat{v}(\theta)}\,d\theta+\sum\limits_{x\in\mathbb{Z}^d}h(x)u(x)v(x)\,=\sum\limits_{x\in\mathbb{Z}^d}f(x,u)v(x).
\end{equation}
For any $y\in \mathbb{Z}^d$, let
\begin{eqnarray*}
	v_0(y)=
	\left\{
	\begin{array}{ll}
		1, &y=x,\\
		0, & y\neq x.
	\end{array}
	\right.
\end{eqnarray*}
By Remark \ref{os}, one sees that $v_0\in H^\alpha.$ Moreover, $\widehat{v_0}(\theta)=\sum\limits_{y\in\mathbb{Z}^d}v_0(y)e^{iy\cdot\theta}\,=e^{ix\cdot\theta}$. Then by (\ref{ve}), we have
\begin{eqnarray*}
\frac{1}{(2\pi)^d}\int_{[-\pi,\pi]^d}(\Phi(\theta))^{\alpha}\hat{u}(\theta)\overline{\widehat{v_0}(\theta)}\,d\theta&=&\frac{1}{(2\pi)^d}\int_{[-\pi,\pi]^d}(\Phi(\theta))^{\alpha}\hat{u}(\theta)e^{-ix\cdot\theta}\,d\theta\\&=&	\frac{1}{(2\pi)^d}\int_{[-\pi,\pi]^d}(\Phi(\theta))^{\alpha}\left(\sum\limits_{y\in\mathbb{Z}^d}u(y)e^{iy\cdot \theta}\right)e^{-ix\cdot\theta}\,d\theta\\&=&\sum\limits_{y\in\mathbb{Z}^d}K^\alpha(x-y)u(y)\,\\&=&(-\Delta)^\alpha u(x).
\end{eqnarray*}
By taking $v=v_0$ in (\ref{oe}) and the previous inequality, we get that
$$(-\Delta)^\alpha u(x)+h(x)u(x)=f(x,u(x)).$$
Hence, $u$ is a solution of the equation (\ref{01}).
\end{proof}

At the end of this section, we give a discrete Lions lemma corresponding to Lions \cite{L1} on $\mathbb{R}^{d}$, which denies a sequence $\{u_n\}$ to distribute itself over on $\mathbb{Z}^d$.

\begin{lm}\label{l-0}
(Lions lemma) Assume that $\{u_n\}$ is bounded in $H^\alpha$ and $\|u_{n}\|_{\infty}\rightarrow0$ as $n\rightarrow+\infty.$
Then for any $2<q<+\infty$, as $n\rightarrow+\infty,$ $$u_n\rightarrow0,\qquad \text{in~} \ell^q(\mathbb{Z}^d).$$

\end{lm}
\begin{proof}
Since $\{u_n\}$ is bounded in $H^\alpha$, $\{u_n\}$ is bounded in $\ell^2(\mathbb{Z}^d)$. For $2<q<+\infty$, this result follows from the interpolation inequality
\begin{eqnarray*}
\|u_n\|^{q}_{q}\leq\|u_n\|_{2}^{2}\|u_n\|_{\infty}^{q-2}.
\end{eqnarray*}

\end{proof}

\section{Nehari manifolds}

In this section, we apply the method of Nehari manifold developed by Szulkin and Weth \cite{SW} to the discrete fractional Schr\"{o}dinger equation (\ref{01}). Since we have established a discrete (Lions) Lemma \ref{l-0},  several modifications would be necessary to deal with our problem.

The Nehari manifold of the functional $I$ corresponding to the equation (\ref{01}) is 
$$\mathcal{M}=\{u\in H^\alpha\backslash\{0\}: \langle I'(u),u\rangle=0\},$$
 where $$\langle I'(u),u \rangle=\|u\|^2_\alpha-\sum\limits_{x\in\mathbb{Z}^d}f(x,u)u(x).$$	
 
 \begin{df}
We say $u\in H^\alpha$ is a ground state solution to the equation (\ref{01}), if 
$u$ is a nontrivial critical point of the functional $I$ satisfying $$
I(u)=\inf\limits_{\mathcal{M}} I>0.$$
\end{df}

Let $S_\rho=\{u\in H^\alpha: \|u\|_\alpha=\rho\},\, \rho>0$. Now we introduce the Nehari's method, which can be seen in \cite{HX,SW}.

\begin{lm}\label{l0} 
Let $(h_1)$ and $(f_1)$-$(f_5)$ hold. Then
\begin{itemize}
\item[(i)] for any $w\in H^\alpha\backslash\{0\},$ there exists a unique $s_w>0$ such that $s_w w\in\mathcal{M}$ and $I(s_w w)=\max\limits_{s>0}I(sw)$;
\item[(ii)] there exists $\delta>0$ such that $s_w\geq\delta$ for each $w\in S_1$; and for each compact subset $\mathcal{K}\subset S_1$, there exists $C_{\mathcal{K}}$ such that $s_w\leq C_{\mathcal{K}}$ for each $w\in\mathcal{K}$.

\end{itemize}
\end{lm}
\begin{proof}
The proof is similar to that of Proposition 3.3 in \cite{HX}, we omit here.	
\end{proof}
The Lemma \ref{l0} (i) states that $s_ww$ is the unique point on the ray $s\mapsto sw, s>0,$ which intersects $\mathcal{M}$;
the first part of (ii) implies that $\mathcal{M}$ is closed in $H^\alpha$ and bounded away from $0$.

\begin{lm}\label{l91}
Let $(h_1)$ and $(f_1)$-$(f_5)$ hold. Then
\begin{itemize}
\item[(i)] there exists $\rho>0$	such that $\inf\limits_{S_\rho}I>0$, and then $c:=\inf\limits_{\mathcal{M}}I\geq\inf\limits_{S_\rho}I>0;$
\item[(ii)] for any $u\in\mathcal{M}$, $\|u\|_\alpha\geq\sqrt{2c}$.
\end{itemize}
\end{lm}
\begin{proof}
(i) Since $p>2$, by Lemma \ref{l01}, we have 
$$\|u\|_p\leq\|u\|_2\leq c_1\|u\|_\alpha.$$
Then, for $u\in H^\alpha$, it follows from (\ref{nn}) that
\begin{eqnarray*}
I(u)&=&\frac{1}{2}\|u\|^2_\alpha-\sum\limits_{x\in\mathbb{Z}^d}F(x,u)\,\\&\geq&\frac{1}{2}\|u\|^2_\alpha-\varepsilon\|u\|^2_2-C_\varepsilon\|u\|^p_p\\&\geq&\frac{1}{2}\|u\|^2_\alpha-\varepsilon\|u\|^2_\alpha-C_\varepsilon\|u\|^p_\alpha.
\end{eqnarray*}
Hence the previous inequality implies that $\inf\limits_{S_\rho}I>0$ holds for $\rho>0$ small enough.

By Lemma \ref{l0} (i), for any $u=s_w w\in \mathcal{M}$, there is $s>0$ such that $sw\in S_\rho$ and $$I(u)\geq I(sw)\geq\inf\limits_{S_\rho}I>0.$$
Hence $c:=\inf\limits_{\mathcal{M}}I\geq\inf\limits_{S_\rho}I>0.$

(ii) By (\ref{30}) and $c=\inf\limits_{\mathcal{M}}I$, for $u\in\mathcal{M}$, we have

$$c\leq I(u)=\frac{1}{2}\|u\|^2_\alpha-\sum\limits_{x\in\mathbb{Z}^d}F(x,u)\,\leq\frac{1}{2}\|u\|^2_\alpha.$$
Hence $\|u\|_\alpha\geq\sqrt{2c}.$ 

\end{proof}

Define the mapping $m:S_1\rightarrow\mathcal{M}$ by setting $m(w):=s_w w$, where $s_w$ is as in Lemma \ref{l0}.
\begin{lm}\label{i}
Let $(h_1)$ and $(f_1)$-$(f_5)$ hold. Then the mapping $m$ is a homeomorphism between $S_1$ and $\mathcal{M}$, and the inverse of $m$ is given by $m^{-1}(u)=\frac{u}{\|u\|_\alpha}$.
\end{lm}
\begin{proof}
The proof of this lemma is essentially the same as for Proposition 3.4 in \cite{HX}, we omit here.	
\end{proof}
Define the functional $\Psi:S_1\rightarrow \R$ by setting $\Psi(w):=I(m(w))$. Since $I\in C^1(H^\alpha,\mathbb{R})$ and $m$ is a continuous mapping, one gets that $\Psi\in C^1(S_1,\mathbb{R}).$ 

Moreover, recall that a sequence $\{u_n\}\subset H^\alpha$ is called
a $(PS)$ sequence of $I$, if $\{I(u_n)\}$ is bounded and $I'(u_n)\rightarrow0$ in $(H^\alpha)^*$, where $(H^\alpha)^*$ is the dual space of $H^\alpha$. If $I(u_n)\rightarrow r$ and $
I'(u_n)\rightarrow 0$ in $(H^\alpha)^*$, we call $\{u_n\}$ a $(PS)_r$ sequence of $I$.

\begin{lm}\label{lm1}
Let $(h_1)$ and $(f_1)$-$(f_5)$ hold. Then
\begin{itemize}
\item[(i)] if $\{w_n\}$ is a $(PS)$ sequence of $\Psi$, then $\{m(w_n)\}$ is a $(PS)$ sequence of $I$; if $\{u_n\}$ is a $(PS)$ sequence of $I$, then $\{m^{-1}(u_n)\}$ is a $(PS)$ sequence of $\Psi$;

\item[(ii)] $w$ is a critical point of $\Psi$ if and only if $m(w)$ is a nontrivial critical point of $I$.

\end{itemize}
\end{lm}

\begin{proof}
By Lemma \ref{vn}, we have the direct sum decomposition
$$H^\alpha=T_w(S_1)\oplus\mathbb{R}w,\quad w\in S_1.$$
Then the following proof of this lemma is the same as for Proposition 3.7 in \cite{HX}, we omit here.	
\end{proof}
\begin{lm}\label{lm3}
Let $(h_1)$ and $(f_1)$-$(f_5)$ hold. Then $I$ is coercive on $\mathcal{M}$.
\end{lm}
\begin{proof}
Let $\{u_n\}\subset\mathcal{M}$ be a sequence such that $I(u_n)\leq r$ uniformly. We claim that $\{u_n\}$ is bounded in $H^\alpha$.

In fact, if $\{u_n\}$ is unbounded in $H^\alpha$, by passing to a subsequence, we may assume that $\|u_n\|_{\alpha}\rightarrow+\infty$ as $n\rightarrow+\infty$. Let $v_n:=\frac{u_n}{\|u_n\|_{\alpha}}\in S_1$. Since $\{v_n\}$ is bounded in $H^\alpha$, we may assume that $v_n\rightharpoonup v $ in $H^\alpha$. By Lemma \ref{l01},  one gets that $\{v_n\}$ is bounded in $\ell^\infty(\mathbb{Z}^d)$. By diagonal principle, there exists a subsequence of $\{v_n\}$ (still denoted by itself) such that $$v_n\rightarrow v, \quad \text{pointwise in~}\mathbb{Z}^d.$$
If $v_n\rightarrow0$ in $\ell^p(\mathbb{Z}^d)$, by (\ref{nn}), for any $s>0$, $\sum\limits_{x\in\mathbb{Z}^d} F(x,sv_n)\,\rightarrow 0$. Hence we have
\begin{equation}\label{v0}
I(sv_n)=\frac{1}{2}s^2-\sum\limits_{x\in\mathbb{Z}^d} F(x,sv_n)\,\rightarrow\frac{1}{2}s^2,\quad n\rightarrow+\infty.	
\end{equation}
By (i) of Lemma \ref{l0}, for any $s>0$,
\begin{equation}\label{v1}
I(sv_n)\leq I(u_n)\leq r.	
\end{equation}
By choosing $s>\sqrt{2r+1}$, we get a contradiction from (\ref{v0}) and (\ref{v1}). Hence $v_n\not\rightarrow 0$ in $\ell^p(\mathbb{Z}^d) $. Since $p>2$, it follows from (Lions) Lemma \ref{l-0} that there exist a sequence $\{y_n\}\subset \mathbb{Z}^d$ and a constant $\delta>0$ such that
\begin{equation}\label{ja}
|v_n(y_n)|\geq \delta.
\end{equation}
Let  $\tilde{v}_{n}(x)=v_n(x+y_n).$ Then up to a subsequence, we have 
\begin{eqnarray*}
\tilde{v}_{n}\rightharpoonup \tilde{v}, \quad \text{in}~H^\alpha,\qquad \text{and~}\qquad
\tilde{v}_{n}\rightarrow \tilde{v}, \quad \text{pointwise~in}~\mathbb{Z}^d.
\end{eqnarray*}
By (\ref{ja}), we have $\tilde{v}(0)\neq0$. 
Note that $\|u_n\|_{\alpha}|\tilde{v}_{n}(0)|\rightarrow+\infty$, by $(f_5)$, we get that
\begin{eqnarray*}
 0\leq\frac{I(u_n)}{\|u_n\|_{\alpha}^2}&=&\frac{1}{2}-\sum\limits_{x\in\mathbb{Z}^d}\frac{F(x,u_n)}{\|u_n\|_{\alpha}^2}\,\\&=&\frac{1}{2}-\sum\limits_{x\in\mathbb{Z}^d}\frac{F(x+y_n,\|u_n\|_{\alpha}\tilde{v}_n)}{\|u_n\|_{\alpha}^2\tilde{v}^2_n}\tilde{v}^2_n\,\\&\leq&\frac{1}{2}-\frac{F(y_n,\|u_n\|_{\alpha}\tilde{v}_n(0))}{\|u_n\|_{\alpha}^2\tilde{v}^2_n(0)}\tilde{v}^2_n(0)\\&\rightarrow &-\infty.
\end{eqnarray*}
This is a contradiction. Hence, $\{u_n\}$ is bounded in $H^\alpha$.

\end{proof}

{\bf Proof of Theorem \ref{t-0}} Let $\{w_n\}\subset S_1$ be a minimizing sequence satisfying $\Psi(w_n)\rightarrow \inf\limits_{S_1}\Psi.$ By the Ekeland variational principle, see \cite{S}, we may assume that $\Psi'(w_n)\rightarrow 0$. Then $\{w_n\}$ is a $(PS)_{\inf\limits_{S_1}\Psi}$ sequence for $\Psi$. Let $u_n=m(w_n)\in\mathcal{M}$. By (i) of Lemma \ref{lm1}, we get that 
\begin{equation}\label{g1}
I'(u_n)\rightarrow 0,\quad n\rightarrow +\infty,	
\end{equation}
and $$I(u_n)=\Psi(w_n)\rightarrow c=\inf\limits_{\mathcal{M}}I=\inf\limits_{S_1}\Psi.$$
By Lemma \ref{lm3}, $\{u_n\}$ is bounded in $H^\alpha$. Up to a subsequence, we assume that
\begin{eqnarray*}
u_n\rightharpoonup u \quad\text{in}~H^\alpha, \qquad \text{and}\qquad  u_n	\rightarrow u,\quad\text{pointwise~in}~\mathbb{Z}^d.
\end{eqnarray*}
Then it follows from (\ref{nn}) that, for any $v\in H^\alpha$,
\begin{eqnarray}\label{g6}
\langle I'(u_n),v\rangle&=&(u_n,v)_\alpha-\sum\limits_{x\in\mathbb{Z}^d} f(x,u_n)v(x)\,\nonumber\\&\rightarrow& (u,v)_\alpha-\sum\limits_{x\in\mathbb{Z}^d} f(x,u)v(x)\,\nonumber\\&=&\langle I'(u),v\rangle.	
\end{eqnarray}
By (\ref{g1}), we get that $I'(u)=0$. In the following, we shall discuss two cases, one is that $u\neq0$, and hence it is a minimizer. The other one is  $u=0$, then by (Lions) Lemma \ref{l-0} and the periodicity of (\ref{01}), we can still find a minimizer. 

{\bf Case 1:} $u\neq 0$. So $u\in\mathcal{M}$ and hence $I(u)\geq c$. Let
\begin{eqnarray*}\label{g2}
g(x,u)=	\frac{1}{2}f(x,u)u-F(x,u).
\end{eqnarray*}
By (\ref{30}), we have $g(x,u)>0.$  Then by Fatou lemma, we get that
\begin{eqnarray*}\label{g3}
\sum\limits_{x\in\mathbb{Z}^d} g(x,u)\,\leq\underset{n\rightarrow+\infty}{\underline{\lim}}\sum\limits_{x\in\mathbb{Z}^d} g(x,u_n)\,.
\end{eqnarray*}
Note that
\begin{eqnarray*}\label{g3}
	c+o(1)=I(u_n)-\frac{1}{2}\langle I'(u_n),u_n\rangle=\sum\limits_{x\in\mathbb{Z}^d} g(x,u_n)\,
\end{eqnarray*}
and
\begin{eqnarray*}
I(u)=I(u)-\frac{1}{2}\langle I'(u),u\rangle=\sum\limits_{x\in\mathbb{Z}^d} g(x,u)\,.
\end{eqnarray*}
Then $I(u)\leq c$. Hence $I(u)=c.$

{\bf Case 2:} $u=0.$ If $u_n\rightarrow0$ in $\ell^p(\mathbb{Z}^d)$, then it follows from (\ref{nn}) and (\ref{g1}) that
\begin{eqnarray*}
\|u_n\|_{\alpha}^2&=&\langle I'(u_n),u_n\rangle+\sum\limits_{x\in\mathbb{Z}^d} f(x,u_n)u_n(x)\,\\&\rightarrow &0,\quad n\rightarrow+\infty.
\end{eqnarray*}
This contradicts the fact that $\{u_n\}\subset\mathcal{M}$ is bounded away from 0, see Lemma \ref{l91} (ii). Hence $u_n\not\rightarrow 0$ in $\ell^p(\mathbb{Z}^d)$. Then by (Lions) Lemma \ref{l-0}, there exist $x_n\in \mathbb{Z}^d$ and $\delta_0>0$ such that
\begin{equation}\label{g5}
|u_n(x_n)|\geq \delta_0.	
\end{equation}
Denote $\tilde{u}_{n}(x)=u_n(x+x_n).$ Then passing to a subsequence if necessary, we assume that
\begin{eqnarray*}
\tilde{u}_{n}\rightharpoonup \tilde{u}, \quad \text{in}~H^\alpha,\qquad \text{and~}\qquad
\tilde{u}_{n}\rightarrow \tilde{u}, \quad \text{pointwise~in}~\mathbb{Z}^d.
\end{eqnarray*}
By (\ref{g5}), we have $\tilde{u}\neq0$. We claim that $\tilde{u}\in\mathcal{M}$.

In fact, by ($f_6$), one sees that $I,\,I'$ and $\mathcal{M}$ are invariant under translation. Therefore, we have $\tilde{u}_n\subset\mathcal{M}$ and 
$$I(\tilde{u}_n)\rightarrow c,\qquad\text{and}\qquad I'(\tilde{u}_n)\rightarrow 0.$$
Then similar to (\ref{g6}), we get that $I'(\tilde{u})=0$. This implies that $\tilde{u}\in\mathcal{M}$. By similar arguments as in {\bf Case 1}, we obtain that $I(\tilde{u})=c$.

 \qed

\section{Multiplicity of solutions}

In this section, we are devoted to looking for infinitely many geometrically distinct solutions for the equation (\ref{01}). From now on, we assume that $f$ is odd in $u$ besides $(h_1)$ and $(f_1)$-$(f_6)$. Note that the results in Section 3 still hold. Recall that $m^{-1}$ is given by
$$m^{-1}:\mathcal{M}\rightarrow S_1,\quad m^{-1}(u)=\frac{u}{\|u\|_\alpha}.$$

\begin{lm}\label{l99}
The map $m^{-1}$ is Lipschitz continuous.	
\end{lm}
\begin{proof}
 For any $u,\,v\in\mathcal{M}$, by Lemma \ref{l91} (ii), we get that 
\begin{eqnarray*}
\left\|m^{-1}(u)-m^{-1}(v)\right\|_\alpha&=&\left\|\frac{u}{\|u\|_\alpha}-\frac{v}{\|v\|_\alpha}\right\|_\alpha	\\&=&\left\|\frac{u-v}{\|u\|_\alpha}+\frac{(v-u)v}{\|u\|_\alpha\|v\|_\alpha}\right\|_\alpha\\&\leq&\frac{2}{\|u\|_\alpha}\|u-v\|_\alpha\\&\leq&\sqrt{\frac{2}{c}}	\|u-v\|_\alpha.
\end{eqnarray*}

\end{proof}
\begin{rem}\label{93}
Note that $m,\,m^{-1},\,I$ and $\Psi$ are invariant with respect to the action of $\mathbb{Z}^d$, then Lemma \ref{lm1} implies that there is a one-to-one correspondence between the orbits $\mathcal{O}(u)\subset\mathcal{M}$ consisting of critical points of $I$ and the orbits $\mathcal{O}(w)\subset S_1$ consisting of critical points of $\Psi$.
\end{rem}

For $r\geq s\geq c$, let
\begin{eqnarray*}
\begin{array}{ll}
I^r:=\{u\in\mathcal{M}: I(u)\leq r\},\quad I_s:=\{u\in\mathcal{M}: I(u)\geq s\},\\\I^r_s:=I^r\cap I_s;\\	
\Psi^r:=\{w\in S_1: \Psi(w)\leq r\},\quad \Psi_s:=\{w\in S_1: \Psi(w)\geq s\},\\\Psi^r_s:=\Psi^r\cap \Psi_s;\\K:=\{w\in S_1: \Psi'(w)=0\},\quad K_r:=\{w\in K: \Psi(w)=r\},\\\nu(r):=\sup\{\|u\|_\alpha: u\in I^r\}.	
\end{array}
\end{eqnarray*}

One gets easily that $\nu(r)<+\infty$ since $I$ is coercive by Lemma \ref{lm3}. We may choose a subset $\mathcal{H}$ of $K$ such that $\mathcal{H}=-\mathcal{H}$ and each orbit $\mathcal{O}(w)\subset K$ has a unique representative in $\mathcal{H}$. By Remark \ref{93}, it suffices to show that the set $\mathcal{H}$ is infinite. Arguing by contradiction, we assume that
\begin{equation}\label{84}
\mathcal{H}~ \text{is a finite set}.	
\end{equation}

\begin{lm}\label{l98}
$k:=\inf\{\|v-w\|_\alpha: v,\,w\in K,\, v\neq w\}>0.$
\end{lm}

\begin{proof}
Let $v_n,\,w_n\in\mathcal{H}$ and $x_n,\,y_n\in\mathbb{Z}^d$ satisfying $$\left\|v_n(\cdot-x_n)-w_n(\cdot-y_n)\right\|_\alpha\rightarrow k,\quad n\rightarrow+\infty,$$
where $v_n(\cdot-x_n)\neq w_n(\cdot-y_n)$ for all $n$.  Denote $z_n=x_n-y_n$. Passing to a subsequence, $v_n=v\in\mathcal{H},\,w_n=w\in\mathcal{H}$ and either $z_n=z\in\mathbb{Z}^d$ for all $n$ or $|z_n|\rightarrow+\infty.$ In the first case,
$$0<\left\|v_n(\cdot-x_n)-w_n(\cdot-y_n)\right\|_\alpha=\left\|v-w(\cdot-z)\right\|_\alpha=k.$$

In the second case, we have $w(\cdot-z_n)\rightharpoonup 0$. Therefore, $$k=\lim\limits_{n\rightarrow+\infty}\left\|v-w(\cdot-z_n)\right\|_\alpha\geq \|v\|_\alpha=1.$$

\end{proof}

The lemma below plays a key role in the discreteness property of the $(PS)$ sequences.

\begin{lm}\label{l88}
Let $r\geq c$. If $\{w^1_n\},\,\{w^2_n\}\subset\Psi^r$ are two $(PS)$ sequences of $\Psi$	, then either $\|w^1_n-w^2_n\|_\alpha\rightarrow0$ or $\underset{n\rightarrow+\infty}{\overline{\lim}}\|w^1_n-w^2_n\|_\alpha\geq\omega(r)>0$, where $\omega(r)$ depends on $r$ but not on the particular choice of the $(PS)$ sequences.

\end{lm}

\begin{proof}
Let $u^1_n:=m(w^1_n)$ and $u^2_n:=m(w^2_n)$. Then $\{u^1_n\},\,\{u^2_n\}\subset I^r$ are the $(PS)$ sequences of $I$ and bounded in $H^\alpha$ since $I$ is coercive on $\mathcal{M}$.

{\bf Case 1:} $\|u^1_n-u^2_n\|_p\rightarrow 0$ as $n\rightarrow+\infty$. Note that $\{u^1_n\},\,\{u^2_n\}\subset I^r$ are bounded $(PS)$ sequences of $I$. Then it follows from (\ref{nn}), Lemma \ref{l01} and H\"{o}lder inequality that, 
\begin{eqnarray*}
\|u^1_n-u^2_n\|^2_\alpha&=&\langle I'(u^1_n),u^1_n-u^2_n\rangle-\langle I'(u^2_n),u^1_n-u^2_n\rangle	\\&&+\sum\limits_{x\in\mathbb{Z}^d} f(x,u^1_n)(u^1_n-u^2_n)\,-\sum\limits_{x\in\mathbb{Z}^d} f(x,u^2_n)(u^1_n-u^2_n)\,\\&\leq&\varepsilon\|u^1_n-u^2_n\|_\alpha+\varepsilon\sum\limits_{x\in\mathbb{Z}^d}
(|u^1_n|+|u^2_n|)|u^1_n-u^2_n|\,\\&&+C_\varepsilon\sum\limits_{x\in\mathbb{Z}^d}
(|u^1_n|^{p-1}+|u^2_n|^{p-1})|u^1_n-u^2_n|\,\\&\leq&(1+C)\varepsilon\|u^1_n-u^2_n\|_\alpha+C\|u^1_n-u^2_n\|_p,\quad n\geq n_\varepsilon.
\end{eqnarray*}
Hence $\|u^1_n-u^2_n\|_\alpha\rightarrow 0$. Then by Lemma \ref{l99}, we get that
$$\|w^1_n-w^2_n\|_\alpha=\left\|m^{-1}(u^1_n)-m^{-1}(u^2_n)\right\|_\alpha\leq\sqrt{\frac{2}{c}}\|u^1_n-u^2_n\|_\alpha\rightarrow 0.$$

{\bf Case 2:} $\|u^1_n-u^2_n\|_p\not\rightarrow 0$ as $n\rightarrow+\infty$.  Since $p>2$, it follows from (Lions) Lemma \ref{l-0} that there exist $y_n\in\mathbb{Z}^d$ and $\delta>0$ such that
\begin{equation}\label{pp}
|(u^1_n-u^2_n)(y_n)|\geq \delta.
\end{equation}
Note that $m,\,m^{-1},\,I'$ and $\Psi'$ are all invariant under the translations of the form $u\mapsto u(\cdot-x)$ with $x\in\mathbb{Z}^d$, we may assume that $\{y_n\}$ is bounded in $\mathbb{Z}^d$. Without loss of generality, we assume that
\begin{eqnarray*}
u^1_n\rightharpoonup u^1,\quad u^2_n\rightharpoonup u^2,\quad
\|u^1_n\|_\alpha\rightarrow \sigma_1,\quad \|u^2_n\|_\alpha\rightarrow \sigma_2.
\end{eqnarray*}
Similar to (\ref{g6}), we have $I'(u^1)=I'(u^2)=0$. By (\ref{pp}), one gets that $u^1\neq u^2.$ Moreover, by Lemma \ref{l91} (ii) and the definition of $\nu$, we have
$$\sqrt{2c}\leq\sigma_i\leq\nu(r),\quad i=1,2.$$

We first consider the case $u^1\neq0$ and $u^2\neq0$. Hence $u^1,\,u^2\in\mathcal{M}$ and $$w^1:=m^{-1}(u^1)\in K,\quad w^2:=m^{-1}(u^2)\in K,\quad w^1\neq w^2.$$
By the definition of $\Psi^r$ and the weak lower semicontinuity of the norm, we have
\begin{eqnarray*}
\underset{n\rightarrow+\infty}{\underline{\lim}}\|w^1_n-w^2_n\|_\alpha&=	&\underset{n\rightarrow+\infty}{\underline{\lim}}\left\|\frac{u^1_n}{\|u^1_n\|_\alpha}-\frac{u^2_n}{\|u^2_n\|_\alpha}\right\|_\alpha\\&\geq&\left\|\frac{u^1}{\sigma_1}-\frac{u^2}{\sigma_2}\right\|_\alpha\\&=&\left\|\beta_1w^1-\beta_2w^2\right\|_\alpha,
\end{eqnarray*}
where $\beta_i=\frac{\|u^i\|_\alpha}{\sigma_i}\geq\frac{\sqrt{2c}}{\nu(r)}, i=1,2.$ Since $\|w^1\|_\alpha=\|w^2\|_\alpha=1$, an elementary geometric argument and the inequalities above imply that
\begin{equation}\label{kk}
\underset{n\rightarrow+\infty}{\underline{\lim}}\|w^1_n-w^2_n\|_\alpha\geq\left\|\beta_1w^1-\beta_2w^2\right\|_\alpha\geq\min\{\beta_1,\beta_2\}\left\|w^1-w^2\right\|_\alpha\geq\frac{k\sqrt{2c}}{\nu(r)}	,
\end{equation}
where $k$ is defined in Lemma \ref{l98}. Let $\omega(r)=\frac{k\sqrt{2c}}{\nu(r)}$ in (\ref{kk}), we get the desired result.

If $u^2=0$, then $u^1\neq 0$ and 
$$\underset{n\rightarrow+\infty}{\underline{\lim}}\|w^1_n-w^2_n\|_\alpha=	\underset{n\rightarrow+\infty}{\underline{\lim}}\left\|\frac{u^1_n}{\|u^1_n\|_\alpha}-\frac{u^2_n}{\|u^2_n\|_\alpha}\right\|_\alpha\geq\left\|\frac{u^1}{\sigma_1}\right\|_\alpha\geq\frac{\sqrt{2c}}{\nu(r)}.$$  

The case $u^1=0$ is treated similarly. The proof is completed.

\end{proof}

Since $\Psi \in C^1(S_1,\mathbb{R})$, by \cite[Lemma II.3.9]{S}, one gets that $\Psi$ admits a pseudo-gradient vector field,  namely, there exists a Lipschitz continuous map $H: S_1\backslash K\rightarrow T(S_1)$ with $H(w)\in T_w(S_1)$ for all $w\in S_1\backslash K$ and
\begin{equation}\label{lc}
\|H(w)\|_\alpha<2\|\nabla \Psi(w)\|_{T^*_w(S_1)},\quad (\nabla \Psi(w), H(w))>\frac{1}{2}\|\nabla \Psi(w)\|_{T^*_w(S_1)}
^2,\quad w\in S_1\backslash K,	
\end{equation}
where the gradient of $\Psi$ at $w\in S_1$ is defined by $$(\nabla\Psi(w),z):=\langle \Psi'(w),z\rangle,\quad z\in T_w(S_1).$$
Let $\eta:\mathcal{G}\rightarrow S_1\backslash K$ be the corresponding flow defined by
\begin{equation}\label{xq}
\left\{
\begin{array}{ll}
\frac{d}{dt}\eta(t,w)=-H(\eta(t,w)),\\
\eta(0,w)=w,	
\end{array}	
\right.
\end{equation}
where $$\mathcal{G}=\{(t,w):w\in S_1\backslash K, T^-(w)<t<T^+(w)\}$$
and $T^-(w)<0, T^+(w)>0$ are the maximal existence times of the trajectory $t\rightarrow\eta(t,w)$ in negative and positive direction. Note that $\Psi$ is strictly decreasing along trajectories of $\eta.$

For deformation-type arguments, the lemma below is crucial.
\begin{lm}\label{lv}
For any $w\in S_1$, the limit $\lim\limits_{t\rightarrow T^+(w)}\eta(t,w)$ exists and is a critical 	point of $\Psi$.
\end{lm}

\begin{proof}
Fix $w\in S_1$ and let $r=\Psi(w).$	

{\bf Case 1:} $T^+(w)<+\infty.$ For $0\leq s<t<T^+(w)$, by (\ref{lc}), (\ref{xq}) and $\inf\limits_{S_1}\Psi=c,$ we have
\begin{eqnarray*}
\|\eta(t,w)-\eta(s,w)\|_\alpha&\leq	&\int^t_s\|H(\eta(t,w))\|_\alpha\,d\tau\\&\leq&2\sqrt{2}\int^t_s\sqrt{(\nabla\Psi(\eta(\tau,w),H(\eta(\tau,w)))}\,d\tau\\&\leq&2\sqrt{2(t-s)}\left(\int^t_s(\nabla\Psi(\eta(\tau,w),H(\eta(\tau,w)))\,d\tau\right)^{\frac{1}{2}}\\&=&2\sqrt{2(t-s)}\int^s_t\frac{d}{d\tau}(\Psi(\eta(\tau,w)))\,d\tau
\\&=&2\sqrt{2(t-s)}[\Psi(\eta(s,w))-\Psi(\eta(t,w))]^{\frac{1}{2}}\\&\leq&2\sqrt{2(t-s)}\left[\Psi(w)-c\right]^{\frac{1}{2}}.
	\end{eqnarray*}
Since $T^+(w)<+\infty$, this implies that $\lim\limits_{t\rightarrow T^+(w)}\eta(t,w)$ exists and then it must be a critical point of $\Psi$ (otherwise the trajectory $t\mapsto\eta(t,w)$ could be continued beyond $T^+(w)$).

{\bf Case 2:} $T^+(w)=+\infty.$  To prove that 
$\lim\limits_{t\rightarrow +\infty}\eta(t,w)$ exists, it suffices to prove that, for any $\varepsilon>0$, there exists $t_\varepsilon>0$ such that
\begin{equation}\label{vc}
\|\eta(t_\varepsilon,w)-\eta(t,w)\|_\alpha<\varepsilon,\quad t\geq t_\varepsilon.	
\end{equation}
If (\ref{vc}) does not hold, then there exist $0<\varepsilon<\frac{1}{2}\omega(r)$ (where $\omega(r)$ is given by Lemma \ref{l88}) and $\{t_n\}\subset[0,+\infty)$ with $t_n\rightarrow+\infty$ and $\|\eta(t_n,w)-\eta(t_{n+1},w)\|_\alpha=\varepsilon$ for all $n$. Choose the smallest $t^1_n\in(t_n,t_{n+1})$ such that 
\begin{equation}\label{xa}
	\|\eta(t^1_n,w)-\eta(t_{n},w)\|_\alpha=\frac{\varepsilon}{3},\quad 0<\varepsilon<\frac{1}{2}\omega(r).
\end{equation}
 Let $\kappa_n=\min\limits_{s\in[t_n,t^1_n]}\|\nabla\Psi(\eta(s,w))\|_{T^*_w(S_1)}
.$
Then by (\ref{lc}) and (\ref{xq}), we have
\begin{eqnarray*}
	\frac{\varepsilon}{3}=\|\eta(t^1_n,w)-\eta(t_{n},w)\|_\alpha&\leq&\int^{t^1_n}_{t_n}\|H(\eta(\tau,w))\|_\alpha\,d\tau\\&\leq&2\int^{t^1_n}_{t_n}\|\nabla\Psi(\eta(\tau,w))\|_{T^*_w(S_1)}
\,d\tau\\&\leq&\frac{2}{\kappa_n}
	\int^{t^1_n}_{t_n}\|\nabla\Psi(\eta(\tau,w))\|^2_{T^*_w(S_1)}
\,d\tau\\&\leq&\frac{4}{\kappa_n}
	\int^{t^1_n}_{t_n}(\nabla\Psi(\eta(\tau,w),H(\eta(\tau,w)))\,d\tau\\&=&\frac{4}{\kappa_n}[\Psi(\eta(t_n,w))-\Psi(\eta(t^1_n,w))].
\end{eqnarray*}
Since $\Psi$ is strictly decreasing along trajectories of $\eta$ and $r=\Psi(w)$, it follows from (\ref{xa}) and Lemma \ref{l88} that $$\Psi((\eta(t_n,w)))-\Psi((\eta(t^1_n,w))\rightarrow0,\quad n\rightarrow+\infty,$$
thus $\kappa_n\rightarrow 0$, and there exist $s^1_n\in[t_n,t^1_n]$ such that $\nabla \Psi(w^1_n)\rightarrow 0$ with $w^1_n=\eta(s^1_n,w)$. Similarly, we find a largest $t^2_n\in (t^1_n, t_{n+1})$ for which $\|\eta(t_{n+1},w)-\eta(t^2_{n},w)\|_\alpha=\frac{\varepsilon}{3}$ and then $w^2_n=\eta(s^2_n,w)$ satisfying $\nabla \Psi(w^2_n)\rightarrow 0$. As $\|w^1_n-\eta((t_n,w))\|_\alpha\leq\frac{\varepsilon}{3}$ and $\|w^2_n-\eta((t_{n+1},w))\|_\alpha\leq\frac{\varepsilon}{3}$, $\{w^1_n\},\{w^2_n\}$ are two $(PS)$ sequences such that
$$\frac{\varepsilon}{3}\leq \|w^1_n-w^2_n\|_\alpha<2\varepsilon<\omega(r).$$
However, this contradicts Lemma \ref{l88}. Hence (\ref{vc}) holds, and thus $\lim\limits_{t\rightarrow +\infty}\eta(t,w)$ exists. Clearly, it must be a critical point of $\Psi.$

\end{proof}

Let $P\subset S_1$, we define
\begin{equation}\label{vb}
U_{\delta}(P)=\{w\in S_1: \text{dist}(w,P)<\delta\},	
\end{equation}
where $\delta>0$ is a constant.

\begin{lm}\label{lg}
Let $r\geq c$. Then for any $\delta>0$, there exists $\varepsilon=\varepsilon(\delta)>0$ such that
\begin{itemize}
	\item [(i)] $\Psi^{r+\varepsilon}_{r-\varepsilon}\cap K=K_r$;
	\item [(ii)] $\lim\limits_{t\rightarrow T^+(w)}\Psi(\eta(t,w))<r-\varepsilon,\quad w\in\Psi^{r+\varepsilon}\backslash U_\delta(K_r).$
\end{itemize}

\end{lm}
\begin{proof}
By (\ref{84}), (i) is true for $\varepsilon>0$ small enough. Without loss of generality, we may assume that $U_\delta(K_r)\subset \Psi^{r+1}$ and $\delta<\omega(r+1)$.
Let 
\begin{equation}\label{xd}
\tau=\inf\{\|\nabla\Psi(w)\|_{T^*_w(S_1)}
: w\in U_\delta(K_r)\backslash U_{\frac{\delta}{2}}(K_r) \}.	
\end{equation}
We claim that $\tau>0.$	 In fact, by contradiction, suppose that there exists a sequence $\{v^1_n\}\subset U_\delta(K_r)\backslash U_{\frac{\delta}{2}}(K_r)$ such that $\nabla \Psi(v^1_n)\rightarrow 0$. Passing to a subsequence, by (\ref{84}) and the $\mathbb{Z}^d-$invariance of $\Psi$, we may assume that $\{v^1_n\}\subset U_\delta(w_0)\backslash U_{\frac{\delta}{2}}(w_0)$ for some $w_0\in K_r.$  Let $v^2_n\rightarrow w_0$, then $\nabla\Psi(v^2_n)\rightarrow 0$ and 
$$\frac{\delta}{2}\leq \underset{n\rightarrow+\infty}{\overline{\lim}}\|v^1_n-v^2_n\|_\alpha\leq \delta<\omega(r+1),$$
	which contradicts Lemma \ref{l88}. Hence $\tau>0.$ Let
	\begin{equation}\label{xf}
	A=\sup\{\|\nabla\Psi(w)\|_{T^*_w(S_1)}
: w\in U_\delta(K_r)\backslash U_{\frac{\delta}{2}}(K_r) \}.	
	\end{equation}
Choose $\varepsilon<\frac{\delta\tau^2}{8A}$	 such that (i) holds.  By Lemma \ref{lv} and (i), the only way (ii) can not be true is that
\begin{equation}\label{70}
	\eta(t,w)\rightarrow\tilde{w}\in K_r,\quad t\rightarrow T^+(w),\quad \text{for~some~}w\in \Psi^{r+\varepsilon}\backslash U_\delta(K_r).
\end{equation}
In this case, let
\begin{eqnarray*}
	\begin{array}{ll}
		t_1=\sup\{t\in[0,T^+(w)): \eta(t,w)\notin U_\delta(\tilde{w})\},\\
		t_2=\inf\{t\in(t_1,T^+(w)): \eta(t,w)\in U_{\frac{\delta}{2}}(\tilde{w})\}.
			\end{array}
\end{eqnarray*}	
Then by (\ref{lc}), (\ref{xq})	and (\ref{xf}), we have
\begin{eqnarray*}
	\frac{\delta}{2}=\|\eta(t_1,w)-\eta(t_2,w)\|_\alpha&\leq&\int^{t_2}_{t_1}\|H(\eta(s,w))\|_\alpha\,ds\\&\leq&2\int^{t_2}_{t_1}\|\nabla\Psi(\eta(s,w))\|_{T^*_w(S_1)}
\,ds\\&\leq&2A(t_2-t_1).
\end{eqnarray*}
By (\ref{lc}), (\ref{xq}), (\ref{xd}) and the inequality above, one gets that
\begin{eqnarray*}
	\Psi(\eta(t_2,w))-\Psi(\eta(t_1,w))&=&-\int^{t_2}_{t_1}(\nabla\Psi(\eta(s,w)),H(\eta(s,w))\,ds\\&\leq&-\frac{1}{2}\int^{t_2}_{t_1}\|\nabla\Psi(\eta(s,w))\|^2_{T^*_w(S_1)}\,ds\\&\leq&-\frac{1}{2}\tau^2(t_2-t_1)\\&\leq&-\frac{\delta \tau^2}{8A}.
\end{eqnarray*}
Hence $\Psi(\eta(t_2,w))\leq r+\varepsilon-\frac{\delta \tau^2}{8A}<r$. As a consequence, $\eta(t,w)\not\rightarrow \tilde{w}$, which contradicts (\ref{70})	.
	
\end{proof}

{\bf Proof of Theorem \ref{t-1} :}
 Let $$\Sigma=\{A\subset S_1: A=-A=\bar{A}\}.$$
For any $A\in\Sigma, \gamma(A)=j$ denotes the usual Krasnoselskii genus of $A$ to be $j$ (see \cite{R}), namely, the smallest integer
$j$ for which there exists an odd mapping $A\rightarrow\mathbb{R}^j\backslash\{0\}$.
In particular, if there does not exist a finite $j$, we set $\gamma(A)=+\infty$. Moreover, we set $\gamma(\emptyset)=0.$

For the usual Krasnoselskii genus, let $A$ and $B$ are closed and symmetric
subsets, then we have the following properties.
\begin{itemize}
	\item [(1)] Mapping property: if there exists an odd map $f\in C(A,B)$, then $\gamma(A)\leq\gamma(B)$.
	\item [(2)] Monotonicity property: if $A\subset B$, then $\gamma(A)\leq\gamma(B)$.
	\item [(3)] Sub-additivity: $\gamma(A\cup B)\leq\gamma(A)+\gamma(B).$
	\item [(4)] Continuity property: if $A$ is compact, then $\gamma(A)<+\infty$ and there is a $\delta>0$
such that $\overline{U_\delta(A)}$ is a closed and symmetric subset, and $\gamma(\overline{U_\delta(A)})=\gamma(A)$,
where $U_\delta(\cdot)$ is defined in (\ref{vb}).
\end{itemize}

Now we consider the nondecreasing sequence of Lusternik-Schnirelman values
for $\Psi$ defined by
$$c_l=\inf\{r\in\mathbb{R}: \gamma(\Psi^r)\geq l, l\in\mathbb{N}\}.$$
Clearly, $c_l\leq c_{l+1}$. We claim that
\begin{equation}\label{ii}
K_{c_l}\neq\emptyset,\quad\text{and}\quad c_l<c_{l+1},\quad l\in\mathbb{N}.
\end{equation}
Indeed, by Lemma \ref{l98}, one sees that $\gamma(K_r)=0$ or $1$  with $r=c_l$ (depending on whether $K_r$ is empty or not.) By the continuity property (4) of the genus, there exists $\delta>0$ such that
\begin{equation}\label{im}
\gamma(\overline{U})=\gamma(K_r)	,
\end{equation}
where $U=U_\delta(K_r)$ and $\delta<\frac{k}{2}$. Choose $\varepsilon=\varepsilon(\delta)>0$ such that Lemma \ref{lg} holds, then for any $w\in \Psi^{r+\varepsilon}\backslash U$, there exists $t\in[0,T^+(w))$ such that $\Psi(\eta(t,w))<r-\varepsilon$. Hence, we can define the entrance map
$$e:\Psi^{r+\varepsilon}\backslash U\rightarrow[0,+\infty),\quad e(w)=\inf\{t\in[0,T^+(w) ): \Psi(\eta(t,w))<r-\varepsilon\},$$
which satisfies $e(w)<T^+(w)$ for $w\in\Psi^{r+\varepsilon}\backslash U$. By Lemma \ref{lg}, one sees that $(r-\varepsilon)$ is not a critical point value of $\Psi$. By $f(x,-t)=-f(x,t)$ for all $(x,t)\in\mathbb{Z}^d\times\mathbb{R}$, (\ref{xq}) and the definition of $\Psi$, we get that $e$ is a continuous and even map. Hence by (\ref{xq}), we have
$$h: \Psi^{r+\varepsilon}\backslash U\rightarrow \Psi^{r-\varepsilon},\quad h(w)=\eta(e(w),w)$$ is odd and continuous. Therefore, by the properties (1)-(3) of the genus and the definition of $r=c_l$, we have
$$\gamma(\Psi^{r+\varepsilon})-\gamma(\overline{U})\leq \gamma(\Psi^{r+\varepsilon}\backslash U)\leq \gamma(\Psi^{r-\varepsilon})\leq l-1,$$
then by (\ref{im}), we get that
$$\gamma(\Psi^{r+\varepsilon})\leq \gamma(K_r)+l-1,$$
namely,
$$\gamma(K_r)\geq \gamma(\Psi^{r+\varepsilon})-(l-1).$$
It follows from the definition of $r=c_l$ and of $c_{l+1}$ that $\gamma(K_r)\geq 1$ if $c_{l+1}>c_l$ and $\gamma(K_r)>1$ if $c_{l+1}=c_l$. Note that $\gamma(\mathcal{H})=\gamma(K_r)\leq 1$, hence (\ref{ii}) holds.

Therefore, (\ref{ii}) implies that there is an infinite sequence $\{\pm w_l\}$ of pairs of
geometrically distinct critical points of $\Psi$ with $\Psi(w_l)=c_l$, which contradicts (\ref{84}). Therefore, the equation (\ref{01}) admits infinitely many pairs of geometrically distinct solutions.

\qed


\
\


\end{document}